\documentclass[reqno]{amsart}
\usepackage{amssymb,amsmath,amsthm,latexsym}
\usepackage{amsrefs,cancel,mathrsfs}

\author{Hyun Chul Jang}
\address{Department of Mathematics, University of Connecticut}
\email{hyun.c.jang@uconn.edu}

\title{Asymptotically Hyperbolic $3$-metric with Ricci flow foliation}

\begin{document}

\newtheorem{theorem}{Theorem} 
\newtheorem{claim}[theorem]{Claim}
\newtheorem{lemma}[theorem]{Lemma}
\newtheorem{proposition}[theorem]{Proposition}
\newtheorem{corollary}[theorem]{Corollary}
\theoremstyle{definition}
\newtheorem{definition}{Definition}
\theoremstyle{remark}
\newtheorem{remark}[theorem]{Remark}

\numberwithin{equation}{section}


\begin{abstract}
In general relativity, there have been a number of successful constructions for asymptotically flat metrics with a certain background foliation. In particular, C. -Y. Lin \cite{Lin:2014bk} used a foliation by the Ricci flow on $2$-spheres to establish an asymptotically flat extension and C. Sormani and Lin \cite{Lin:2016ca} proved useful results with this extension. In this paper, we construct asymptotically hyperbolic $3$-metrics with the Ricci flow foliation inspired by Lin's work.  We also study the rigid case when the Hawking mass of the inner surface of the manifold agrees with its total mass as in \cite{Lin:2016ca}.
\end{abstract}

\maketitle

\section{Introduction}\label{introduction}
In general relativity, the vacuum Einstein field equations of a spacetime $(M,\gamma)$ with cosmological constant $\Lambda$ can be written as
\[\text{Ric}_\gamma=\Lambda \gamma.\]
When $\Lambda<0$ the lowest energy solution is called the anti-de Sitter(AdS) spacetime. The significance of AdS spacetimes has increased especially from the AdS/CFT correspondence. In this context, asymptotically hyperbolic Riemannian 3-manifolds arise naturally as spacelike hypersurfaces of AdS spacetimes. Moreover, these manifolds can be considered spacelike hypersurfaces in asymptotically flat spacetimes which approach null infinity.

The construction of asymptotically flat solutions to the Einstein constraint equations, which provide Cauchy data for Einstein equation with $\Lambda=0$, has been studied extensively. From physical motivation, the dominant energy condition requires the scalar curvature of such metrics to be nonnegative. Due to this condition, R. Bartnik \cite{Bartnik:1993er} introduced a construction of $3$-metrics with prescribed scalar curvature by considering $3$-manifolds foliated by round spheres. There have been other interesting results inspired by this foliation construction. See \cite{Shi:2002dv, Smith:2004wp, Lin:2014bk}. In particular, C. -Y. Lin \cite{Lin:2014bk} used Hamilton's modified Ricci flow on surfaces as foliation to construct an asymptotically flat end. Let $(\Sigma,g)$ be a surface diffeomorphic to $\mathbb{S}^2$ whose area is $4\pi$. Recall that Hamilton's modified Ricci flow in \cite{Hamilton:1988bp} is defined as the family $(\Sigma,g(t))$ satisfying
\begin{equation}\label{RF}
\left\{
\begin{split}
\frac{\partial}{\partial t}g_{ij}&=(r-R)g_{ij}+2D_iD_j f=2M_{ij},\\
g(1)&=g,
\end{split}\right.
\end{equation} 
where $R=R(t)$ is the scalar curvature of $g(t)$ on $\Sigma$ and
$$r=\frac{1}{|\Sigma_{t}|}\int_\Sigma R(t)\,d\mu_t=2,$$
where $|\Sigma_{t}|$ is the area of $\Sigma$ with respect to $g(t)$ and $f=f(t,x)$ is the Ricci potential satisfying the equation 
$$\Delta f=R-r.$$
Note that this flow converges to a metric of constant curvature exponentially fast in any $C^k$-norm (see \cite[Appendix B]{Chow:vq}). Consider a metric $\overline{g}$ on $N=[1,\infty)\times \Sigma$ of the form
$$\overline{g}=u^2 dt^2 +t^2 g(t).$$
The unknown function $u$ on $N$ with prescribed scalar curvature $\overline{R}$ satisfies a quasilinear second order parabolic equation derived from the Gauss equation for each slice $\{t\}\times\Sigma$. Therefore, Lin obtained asymptotically flat $3$-metrics by solving this equation with some conditions for $\overline{R}$. Moreover, C. Sormani and Lin \cite{Lin:2016ca} studied the class of asymptotically flat three-dimensional Riemannian manifolds foliated by Hamilton's modified Ricci flow, and they used these manifolds to estimate the Bartnik mass. In addition, they showed rigidity and monotonicity of the Hawking mass of level sets given in the foliation (see \cite[Theorem 5]{Lin:2016ca}). 

In this paper, we construct an asymptotically hyperbolic $3$-metric using Ricci flow foliation method and investigate properties of the metric. First we recall the most general form of definition of asymptotically hyperbolic manifolds in \cite{Chrusciel:2003fg}. Let $\mathbb{H}^n$ denote the standard hyperbolic space. The metric $g_0$ on $\mathbb{H}^n$ can be written as
\[g_0=dr^2+(\sinh r)^2 h_0,\]
where $r$ is the $g_0$-distance to a fixed point $o$ and $h_0$ is the standard metric on $\mathbb{S}^{n-1}$. Let $\varepsilon_0=\partial_r,\varepsilon_\alpha=\frac{1}{\sinh\rho}f_\alpha,\alpha=1,\ldots,n-1$, where $\{f_\alpha\}_{1\leq\alpha\leq n-1}$ is a local orthonormal frame of $(\mathbb{S}^{n-1},h_0)$ so that $\{\varepsilon_i\}_{0\leq i\leq n-1}$ forms a local orthonormal frame on $\mathbb{H}^n$. Let $S_r$ denote the geodesic sphere in $\mathbb{H}^n$ of radius $r$ centered at $o$.
\begin{definition}\label{defn AH}
	A manifold $(M^n,g)$ is called \emph{asymptotically hyperbolic} if, outside a compact set, $M$ is diffeomorphic to the exterior of some geodesic sphere $S_{r_0}$ in $\mathbb{H}^n$ such that the metric components $g_{ij}=g(\varepsilon_i,\varepsilon_j),0\leq i,j\leq n-1$, satisfy
	\[|g_{ij}-\delta_{ij}|=O(e^{-\tau r}),|\varepsilon_k(g_{ij})|=O(e^{-\tau r}),|\varepsilon_k(\varepsilon_l(g_{ij}))|=O(e^{-\tau r})\]
	for some $\tau>\frac{n}{2}$.
\end{definition}
There are several other versions of the definition with different contexts, see \cite{Chen:2016hm}, \cite{Wang:2001hg}, \cite{Neves:2010hb}.

In section \ref{construction}, we derive the quasilinear parabolic equation for a function $u$ from prescribed scalar curvature $\overline{R}$ on $N$, and prove the existence of a solution. To estimate $C^0$ bounds, we introduce the substitution $w=u^{-2}$ which provides an explicit form of the bounds. (c.f. Remark \ref{substitution remark})

In section \ref{metric section}, we construct an asymptotically hyperbolic $3$-metric by the result in section \ref{construction}. 
\begin{theorem}\label{theorem 1}
	Let $(\Sigma,g)$ be a $2$-manifold which is diffeomorphic to $\mathbb{S}^2$ with area $4\pi$ and let $N$ be the product manifold $[1,\infty)\times\Sigma$. Then for any $H\in C^{\infty}(\Sigma)$ with $H>0$, there exists an asymptotically hyperbolic $3$-metric on $N$ of the form
	\begin{equation}\label{metric}
	\overline{g}=\frac{u^2}{1+t^2}dt^2+t^2g(t),
	\end{equation}
	with the scalar curvature $\overline{R}\equiv-6$ where $u\in C^\infty(N)$ is positive everywhere, and $g(t)$ is the solution to Hamilton's modified Ricci flow (\ref{RF}). Here $H$ is the mean curvature in direction $\partial_{t}$ on $\{1\}\times\Sigma$.
\end{theorem}
As in \cite{Lin:2014bk}, the crucial step here is to verify when the solution of the equation in section \ref{construction} exists. In fact, we can construct an asymptotically hyperbolic $3$-metric with a more general condition on the scalar curvature as
\[\overline{R}=-6+O(t^{-5})\geq -6.\]
The dominant energy condition requires that $\overline{R}\geq -6$, and the decay is needed to control the behavior of $u$ near infinity. See Theorem \ref{main theorem}.
 
In section \ref{hawking mass section}, we prove the corresponding rigidity and monotonicity result of the hyperbolic analogue of the Hawking mass as in \cite[Theorem 5]{Lin:2016ca}. 
The mass of asymptotically hyperbolic Riemannian manifolds is defined as a linear functional by P. T. Chru\'sciel and M. Herzlich \cite{Chrusciel:2003fg}.
\begin{definition}
	The mass functional $\mathbf{M}(g)$ is a linear functional on the kernel of $(\mathcal{DS})^*_{g_0}$, where $(\mathcal{DS})^*_{g_0}$ is the formal adjoint of the linearization of the scalar curvature at $g_0$. Let $\theta=(\theta^1,\ldots,\theta^n)\in \mathbb{S}^{n-1}\subset\mathbb{R}^n$ and the functions
	\[V^{(0)}=\cosh r,\, V^{(j)}=\theta^j\sinh r\text{ for }1\leq j\leq n.\]
	forms a basis of the kernel of $(\mathcal{DS})^*_{g_0}$. Using this basis, $\mathbf{M}(g)$ is defined as
	\[\begin{split}
	\mathbf{M}(g)(V^{(i)})&=\frac{1}{2(n-1)\omega_{n-1}}\lim_{r\rightarrow\infty}\int_{S_r}\left[V^{(i)}(\text{div}_0 h-d\text{tr}_0 h)\right.\\
	&\left.\hspace{2in}-h(\nabla_0 V^{(i)},\cdot)+(\text{tr}_0 h)dV^{(i)}\right](\nu_0)d\sigma_0
	\end{split}\]
	where $h=g-g_0,\nu_0$ is the $g_0$-unit outward normal to $S_r$ and $d\sigma_0$ is the volume element on $S_r$ of the metric induced from $g_0$. The notation $\text{div}_0,\text{tr}_0,\nabla_0$ denotes the divergence, trace, the covariant derivative with respect to $g_0$, respectively. 
\end{definition}
The Hawking mass on asymptotically hyperbolic manifolds is introduced by X. Wang \cite{Wang:2001hg}. (see also \cite{Neves:2010hb}).
\begin{definition}
	Let $(M^3,g)$ be a $3$-dimensional asymptotically hyperbolic manifold, and let $\Sigma\subset M^3$ be a closed $2$-surface. Then the Hawking mass $\tilde{\mathfrak{m}}_H(\Sigma)$ of $\Sigma$ is defined as 
	\begin{equation}\label{hawking mass definition}
	\tilde{\mathfrak{m}}_H(\Sigma)=\sqrt{\frac{|\Sigma|}{16\pi}}\left(1-\frac{1}{16\pi}\int_{\Sigma}H^2\, d\sigma+\frac{|\Sigma|}{4\pi}\right)
	\end{equation}
	where $d\sigma$ is the induced volume form with respect to $g$.
\end{definition}
To study the rigid case, we use the following result about the Hawking mass and the mass functional
\begin{equation}\label{hawking mass relation}
\mathbf{M}(g)(V^{(0)})=\lim_{r\rightarrow\infty}\tilde{\mathfrak{m}}_H(S_r)
\end{equation}
proved by P. Miao, L. -F. Tam, and N. Xie \cite{Miao:2017df}.
\begin{theorem}\label{rigidity theorem}
	Let $(\Sigma,g_1)$ be a surface diffeomorphic to $\mathbb{S}^2$ with positive mean curvature (not necessarily constant) and let $N=[1,\infty)\times\Sigma$ be an asymptotically hyperbolic extension obtained in Theorem \ref{main theorem}. Then $\tilde{\mathfrak{m}}_H(\Sigma_{t})$ is nondecreasing, where $\Sigma_{t}=\{t\}\times\Sigma$. Furthermore, if 
	\begin{equation}
	\mathbf{M}(\overline{g})(V^{(0)})=\tilde{\mathfrak{m}}_H(\Sigma)
	\end{equation}
	then $\overline{R}=-6$ everywhere, $\Sigma$ is isometric to the standard sphere, and $N$ is rotationally symmetric. If $\tilde{\mathfrak{m}}_H(\Sigma)=0$ then $N$ is isometric to a rotationally symmetric region in a hyperbolic space. If $\tilde{\mathfrak{m}}_H(\Sigma)=m>0$ then $N$ is isometric to a rotationally symmetric region in anti-de Sitter Schwarzschild space of mass $m$.
\end{theorem}
\subsection*{Acknowledgments} The author would like to express my gratitude to my advisor, Professor Lan-Hsuan Huang, for all her support, guidance, and motivation. The author is also grateful to Professors Xiadong Yan, Ovidiu Munteanu and Guozhen Lu for their helpful discussions. The author was partially supported by the NSF under grant DMS 1452477.

\section{Parabolic equation with Ricci flow foliation}\label{construction}
In this section, we will derive the equation for prescribed scalar curvature $\overline{R}$ on $N$ from (\ref{metric}) and obtain a priori estimates for a solution $u$. The argument is slightly modified from \cite{Bartnik:1993er} to be suitable for the derived equation. 

From the Gauss equation for each slice $\{t\}\times \Sigma$, we have
$$\overline{R}=R_t+2\overline{\text{Ric}}\left(\frac{\sqrt{1+t^2}}{u}\partial_t,\frac{\sqrt{1+t^2}}{u}\partial_t \right)+||h||^2-H^2,$$
where $R_t$ is the scalar curvature on $\{t\}\times\Sigma$ with the induced metric $t^2g(t)$, $h$ is the second fundamental form, and $H$ is the mean curvature in direction $\partial_t$. By direct computation, we have
\[\begin{split}\overline{\Gamma}_{ij}^{0}  =\frac{1}{2}\overline{g}^{0l}(\overline{g}_{lj,i}+\overline{g}_{il,j}-\overline{g}_{ij,l})&=\frac{1+t^{2}}{2u^{2}}\left(-\frac{\partial}{\partial t}(t^{2}g(t)_{ij})\right)\\
	& =\frac{1+t^{2}}{u^{2}}\left(-2tg_{ij}-2t^{2}M_{ij}\right),\\
	h_{ij}  =-\overline{g}\left(\overline{\nabla}_{\partial_{i}}\partial_{j},\frac{\sqrt{1+t^{2}}}{u}\partial_{0}\right)&=-\frac{\sqrt{1+t^{2}}}{u}\overline{g}(\overline{\Gamma}_{ij}^{0}\partial_{0},\partial_{0})\\
	& =\frac{\sqrt{1+t^{2}}}{2u}(2tg_{ij}+2t^{2}M_{ij}),\\\end{split}
\]\[\begin{split}
	H  &=\overline{g}^{ij}h_{ij}=t^{-2}g^{ij}\left(\frac{\sqrt{1+t^{2}}}{u}tg_{ij}\right)=\frac{2\sqrt{1+t^{2}}}{tu},\\
	||h||^{2} & =\frac{2(1+t^{2})}{t^{2}u^{2}}+\frac{1+t^{2}}{u^{2}}|M|_{g(t)}^{2}.
\end{split}
\]
Then, we obtain
\[\overline{\text{Ric}}\left(\frac{\sqrt{1+t^{2}}}{u}\partial_{t},\frac{\sqrt{1+t^{2}}}{u}\partial_{t}\right)=-\frac{1}{u}\Delta_{\overline{g}|_{\Sigma_{t}}}u+\frac{\sqrt{1+t^{2}}}{u}\frac{\partial H}{\partial t}-||h||^{2},\]
\[\begin{split}
h_{ij} & =\frac{\sqrt{1+t^{2}}}{u}(tg_{ij}+t^{2}M_{ij}),\\
H & =\frac{2\sqrt{1+t^{2}}}{tu},\quad ||h||^{2}=\frac{2(1+t^{2})}{t^{2}u^{2}}+\frac{1+t^{2}}{u^{2}}|M|_{g(t)}^{2}.
\end{split}
\]
Thus we get the following quasilinear second order parabolic equation
\begin{equation}\label{targetPDE}
\begin{split}
t(1+t^{2})\frac{\partial u}{\partial t}=\frac{u^{2}\Delta_{g(t)}u}{2}&-\frac{u^{3}}{4}(R_{g(t)}-t^{2}\overline{R})\\
&+u\left(\frac{1+3t^{2}}{2}+\frac{t^{2}(1+t^{2})|M|_{g(t)}^{2}}{4}\right).
\end{split}
\end{equation}
Here $\Delta_{g(t)}$ and $R_{g(t)}$ are the Laplace operator and the scalar curvature on $(\{t\}\times\Sigma,g(t))$ respectively. For any interval $I\subset\mathbb{R}^{+}$, let $A_{I}=I\times\Sigma$.
For sake of convenience, we will use the following notations: $\Delta=\Delta_{g(t)},$ $\nabla=\nabla^{g(t)}$, for any $f\in C^{0}(A_{I})$, $f^{*},f_{*}:I\rightarrow\mathbb{R}$ are defined by
\[f_{*}(t)=\inf\{f(t,x):x\in\Sigma\},\qquad f^{*}(t)=\sup\{f(t,x):x\in\Sigma\}.\]

Now from the parabolicity of (\ref{targetPDE}), the local existence can be obtained by standard Schauder theory \cite[Theorem 8.2]{Lieberman:1996ue}.

\begin{proposition}\label{proposition short time}
    Let $I=[t_{0},t_{1}],1\leq t_{0}<t_{1}<\infty$, and let $\overline{R}\in C^{\alpha,\alpha/2}(A_{I})$. Then for any initial condition 
	\begin{equation}\label{initial}
	u(t_{0},x)=\varphi(x),\quad x\in\Sigma,
	\end{equation}
	where $\varphi\in C^{2,\alpha}(\Sigma)$ satisfies 
	\begin{equation}\label{initialP}
	0<\delta_{0}\leq\varphi^{-2}(x)\leq\delta_{0}^{-1},\quad x\in\Sigma,
	\end{equation}
	for some constant $\delta_{0}>0$, the parabolic equation (\ref{targetPDE}) with the initial condition (\ref{initial}) has a unique solution $u\in C^{2+\alpha,1+\alpha/2}(A_{[t_{0},t_{0}+T]})$ for some $T>0$. Here $T$ depends on $\delta_{0},t_{0},||R||_{\alpha,\alpha/2;A_I}, ||M||_{\alpha,\alpha/2;A_I},$ $||\overline{R}||_{\alpha,\alpha/2;A_I}$ and $||\varphi||_{2,\alpha}$.
\end{proposition}
To state the existence of global solution, we need the following a priori $C^0$ estimates for the solution $u$ which control the parabolicity and prevent the finite-time blow up.
\begin{proposition}\label{proposition for delta}
	Suppose $u\in C^{2+\alpha,1+\alpha/2}(A_{[t_{0},t_{1}]}),1\leq t_{0}<t_{1}$, is a positive solution to (\ref{initial}). If we further assume that $\overline{R}$ is defined on $A_{[1,\infty)}$ such that the functions 
	\begin{equation}\label{lower delta}
	\delta_{*}(t)=\frac{1}{t(1+t^{2})}\int_{1}^{t}\frac{\left(R_{g(s)}-s^{2}\overline{R}\right)_{*}}{2}\exp\left(-\int_{s}^{t}\frac{\tau(|M|^{*})^{2}}{2}d\tau\right)ds
	\end{equation}
	and 
	\begin{equation}\label{upper delta}
	\delta^{*}(t)=\frac{1}{t(1+t^{2})}\int_{1}^{t}\frac{\left(R_{g(s)}-s^{2}\overline{R}\right)^{*}}{2}\exp\left(-\int_{s}^{t}\frac{\tau(|M|_{*})^{2}}{2}d\tau\right)ds
	\end{equation}
	are defined and finite for all $t\in[t_{0},\infty)$, then for $t_{0}\leq t\leq t_{1}$, we have
	\begin{equation}\label{lower bound}
	u^{-2}(t,x)\geq\delta_{*}(t)+\frac{t_{0}(1+t_{0}^{2})}{t(1+t^{2})}(u^{*}(t_{0})^{-2}-\delta_{*}(t_{0}))\exp\left(-\int_{t_{0}}^{t}\frac{\tau(|M|^{*})^{2}}{2}d\tau\right)
	\end{equation}
	and 
	\begin{equation}\label{upper bound}
	u^{-2}(t,x)\leq\delta^{*}(t)+\frac{t_{0}(1+t_0^{2})}{t(1+t^{2})}(u_{*}(t_{0})^{-2}-\delta^{*}(t_{0}))\exp\left(-\int_{t_{0}}^{t}\frac{\tau(|M|_{*})^{2}}{2}d\tau\right).
	\end{equation}
\end{proposition}
\begin{proof}
	Let $w=u^{-2}$ then we have
	\[\Delta u 
	=-\frac{3}{2}w^{-1}\nabla u\cdot\nabla w-\frac{1}{2}w^{-\frac{3}{2}}\Delta w.
	\]
	Substituting the Laplace term in (\ref{targetPDE}), we obtain
	\begin{equation}\label{after substitution}
	\begin{split}
	\frac{\partial w}{\partial t} =\frac{1}{t(1+t^{2})}\left[\frac{3}{2}u\nabla u\cdot\nabla w+\frac{1}{2w}\right.&\Delta w+\frac{1}{2}(R_{g(t)}-t{}^{2}\overline{R})\\
	&\left.-w\left(1+3t^{2}+\frac{t^{2}(1+t^{2})|M|^{2}}{2}\right)\right].
	\end{split}
	\end{equation}
	By applying the maximum principle, we have
	\[t\frac{dw_{*}}{dt}\geq\frac{1}{1+t^{2}}\left(-w_{*}\left(1+3t^{2}+\frac{t^{2}(1+t^{2})\left(|M|^{*}\right)^{2}}{2}\right)+\frac{1}{2}(R_{g(t)}-t{}^{2}\overline{R})_{*}\right)\] at the maximum of $u(t,x)$. We solve the following ODE
	\begin{equation}\label{ODE maximum principle}
	t\frac{dw_{*}}{dt}=-w_{*}\left(1+t\left(\frac{2t}{1+t^{2}}+\frac{t\left(|M|^{*}\right)^{2}}{2}\right)\right)+\frac{(R_{g(t)}-t^{2}\overline{R})_{*}}{2(1+t)^{2}}.
	\end{equation}
	Using the integrating factor method, we let \[\varphi(t)=\exp\left(\int_{1}^{t}\frac{2s}{1+s^{2}}+\frac{s\left(|M|^{*}\right)^{2}}{2}ds\right).\] Then we have
	\[\frac{d(t\varphi(t)w_{*}(t))}{dt}=\frac{(R_{g(t)}-t^{2}\overline{R})_{*}}{2(1+t^{2})}\varphi(t).\] Integrating to solve $t\varphi(t)w_*(t)$ and noting $u^{-2}(t,x)=w(t,x)\geq w_*(t)$, we derive
	\[\begin{split}u^{-2} & \geq\frac{1}{t}\int_{t_{0}}^{t}\frac{\left(R_{g(s)}-s^{2}\overline{R}\right)_{*}}{2(1+s^{2})}\,\exp\left(-\int_{s}^{t}\left(\frac{2\tau}{1+\tau^{2}}+\frac{\tau(|M|^{*})^{2}}{2}\right)d\tau\right)ds\\
	&\qquad\qquad\qquad\qquad\qquad +\frac{t_{0}}{t}\exp\left(-\int_{t_{0}}^{t}\left(\frac{2\tau}{1+\tau^{2}}+\frac{\tau(|M|^{*})^{2}}{2}\right)d\tau\right)w_{*}(t_{0})\\
	&=\delta_{*}(t)+\frac{t_{0}(1+t_0^2)}{t(1+t^2)}(w_{*}(t_{0})-\delta_{*}(t_{0}))\exp\left(-\int_{t_{0}}^{t}\left(\frac{\tau(|M|^{*})^{2}}{2}\right)d\tau\right).
	\end{split}\]
	Similarly, applying the maximum principle to $w^{*}$, we get the upper bound of $u^{-2}$.
\end{proof}
\begin{remark}\label{substitution remark}
	In proof of Proposition \ref{proposition for delta}, the idea of substitution as $w=u^{-2}$ first appeared in Bartnik's work \cite[Proposition 3.3]{Bartnik:1993er}. The advantage of this is that we can simplify the coefficient of the term $R_{g(t)}-t^2\overline{R}$ as in (\ref{after substitution}) so that we can find the explicit solution of the equation (\ref{ODE maximum principle}) when applying the maximum principle. This will be used not only to prove the global existence of solution but also to show that $\overline{g}=\frac{u^2}{1+t^2}dt^2+t^2g(t)$ is asymptotically hyperbolic with a certain initial condition on $\overline{R}$. 
\end{remark}
	
With Propositions \ref{proposition short time} and \ref{proposition for delta}, we can prove the global existence of solution as the following.
\begin{theorem}\label{global existence theorem}
	Assume that $\overline{R}\in C^{\alpha,\alpha/2}(N)$ and the constant $K$ is defined by
	\begin{equation}\label{constant K}
	\begin{split}K =\sup_{1\leq t<\infty}\left\{ -\int_{1}^{t}\frac{\left(R_{g(s)}-s^{2}\overline{R}\right)_{*}}{4}\exp\left(\int_{1}^{s}\frac{\tau(|M|^{*})^{2}}{2}d\tau\right)ds\right\} <\infty.
	\end{split}
	\end{equation}
	Then for every $\varphi\in C^{2,\alpha}(\Sigma)$ such that 
	\begin{equation}\label{global condition}
	0<\varphi(x)<\frac{1}{\sqrt{K}}\text{ for all }x\in\Sigma,
	\end{equation}
	there is a unique positive solution $u\in C^{2+\alpha,1+\alpha/2}(N)$ with the initial condition 
	\begin{equation}\label{target PDE initial}
	u(1,\cdot)=\varphi(\cdot).
	\end{equation}
\end{theorem}
\begin{proof}
	By considering (\ref{lower bound}) and (\ref{constant K}) simultaneously, we have
	\[\begin{split}
	\left(u^{-2}\right)_*(t)&>\delta_{*}(t)+\frac{2K}{t(1+t^2)}\exp\left(-\int_{1}^t\frac{\tau\left(|M|^*\right)^2}{2}d\tau\right)\\
	&\geq\frac{2}{t(1+t^{2})}\exp\left(-\int_{1}^t\frac{\tau\left(|M|^*\right)^2}{2}d\tau\right)\\
	&\hspace{1in}\times\left(\int_{1}^{t}\frac{\left(R_{g(s)}-s^{2}\overline{R}\right)_{*}}{4}\exp\left(\int_{1}^{s}\frac{\tau(|M|^{*})^{2}}{2}d\tau\right)ds+K\right)\\
	&\geq 0
	\end{split}
	\]
	for all $t\geq 1$. Hence it follows from Proposition \ref{proposition for delta} that $u$ doesn't blow up for all $t\geq 1$. Combining this and Proposition \ref{proposition short time} which states the local existence, we get the desired result.
\end{proof}

\section{Asymptotically hyperbolic $3$-metric with Ricci flow foliation}\label{metric section}
Using Theorem \ref{global existence theorem}, we can construct a metric with prescribed scalar curvature $\overline{R}$ along the Ricci flow foliation. By assuming the approximate decay for $\overline{R}$, we prove that the metric is asymptotically hyperbolic.

\begin{theorem}\label{main theorem}
	Let $(\Sigma,g)$ be a $2$-manifold which is diffeomorphic to $\mathbb{S}^2$ with area $4\pi$. Let $N$ be the product manifold $[1,\infty)\times\Sigma$. Assume that $\overline{R}\in C^{\infty}(N)$ satisfies
	\begin{equation}\label{source condition}
	\overline{R}=-6+O(t^{-5})\geq -6. 
	\end{equation}
	Then for any $H\in C^{\infty}(\Sigma)$ with the condition
	\begin{equation}\label{extension condition}
	H>2\sqrt{2K}
	\end{equation}
	where $K$ is defined as (\ref{constant K}), there exists an asymptotically hyperbolic $3$-metric on $N$ of the form
	\begin{equation}
	\overline{g}=\frac{u^2}{1+t^2}dt^2+t^2g(t).
	\end{equation}
	Here $g(t)$ is the solution to Hamilton's modified Ricci flow (\ref{RF}), such that $\overline{R}$ and $H$ are the scalar curvature on $(N,\overline{g})$ and the mean curvature in direction $\partial_{t}$ on $\{1\}\times\Sigma$, respectively.
\end{theorem}

To prove the theorem we need the following lemma which investigates the decay of $u$ by the assumption (\ref{source condition}). The method is similar to the proof of Y. Shi and L. -F. Tam in \cite{Shi:2002dv}.

\begin{lemma}\label{decay Lemma}
	Let $u$ be the solution of (\ref{targetPDE}) with the initial condition $u(1,x)=\varphi(x)$, where $\varphi(x)$ satisfies (\ref{global condition}). Then for sufficiently large $t$, we have the estimate 
	\begin{equation}
	\begin{split}
	\left|\left(\frac{\partial}{\partial t}\right)^k\left(\frac{\partial^{|\beta|}}{\partial x^\beta}\right)(u-1)\right|\leq \frac{C}{t^3}
	\end{split}
	\end{equation}
	where $\beta$ is a multi-index.
\end{lemma}
\begin{proof}[Proof of Lemma \ref{decay Lemma}]
	First we need to verify the $C^0$ bounds. Since $|R_{g(s)}-2|$ is bounded we have
	\[\begin{split}
	&\int_{1}^{t}\frac{(R_{g(s)}-2)^*}{2}\exp\left(-\int_{s}^{t}\frac{\tau(|M|_{*})^{2}}{2}d\tau\right)ds\leq C_1,\\
	&\int_{1}^{t}\frac{(R_{g(s)}-2)_*}{2}\exp\left(-\int_{s}^{t}\frac{\tau(|M|^*)^2}{2}d\tau\right)ds\geq C_2.
	\end{split}\]
	From the scalar curvature condition (\ref{source condition}), we obtain 
	\[\begin{split}\delta^{*}(t) & =\frac{1}{t(1+t^{2})}\int_{1}^{t}\frac{\left(R_{g(s)}-s^{2}\overline{R}\right)^{*}}{2}\exp\left(-\int_{s}^{t}\frac{\tau(|M|_{*})^{2}}{2}d\tau\right)ds\\
	& =\frac{1}{t(1+t^{2})}\int_{1}^{t}\exp\left(-\int_{s}^{t}\frac{\tau(|M|_{*})^{2}}{2}d\tau\right)ds\\
	& \qquad+\frac{1}{t(1+t^{2})}\int_{1}^{t}\frac{\left(R_{g(s)}-2+6s^{2}+O(s^{-3})\right)^{*}}{2}\exp\left(-\int_{s}^{t}\frac{\tau(|M|_{*})^{2}}{2}d\tau\right)ds.
	\end{split}\]
	Since $\exp\left(-\int_{s}^{t}\frac{\tau(|M|_{*})^{2}}{2}d\tau\right)\leq 1$, we have
	\[
	\int_{1}^{t}\exp\left(-\int_{s}^{t}\frac{\tau(|M|_{*})^{2}}{2}d\tau\right)\leq t-1.
	\]
	Thus we get
	\[\begin{split}
	\delta^*(t)&\leq \frac{1}{1+t^2}+\frac{C_1-1}{t(1+t^2)}+1-\frac{1}{1+t^2}+O(t^{-3})\\
	&\leq 1+\frac{C_3}{t^3}.
	\end{split}
	\]
	To get the lower bound estimate, similarly we have
	\[\begin{split}\delta_{*}(t) & =\frac{1}{t(1+t^{2})}\int_{1}^{t}\frac{\left(R_{g(s)}-s^{2}\overline{R}\right)_{*}}{2}\exp\left(-\int_{s}^{t}\frac{\tau(|M|^{*})^{2}}{2}d\tau\right)ds\\
	& =\frac{1}{t(1+t^{2})}\int_{1}^{t}\exp\left(-\int_{s}^{t}\frac{\tau(|M|^{*})^{2}}{2}d\tau\right)ds\\
	& \qquad+\frac{1}{t(1+t^{2})}\int_{1}^{t}\frac{\left(R_{g(s)}-2+6s^{2}+O(s^{-3})\right)_{*}}{2}\exp\left(-\int_{s}^{t}\frac{\tau(|M|^{*})^{2}}{2}d\tau\right)ds.\\
	\end{split}\]
	Let $t_{0}\geq 1$ such that \[\int_{t_{0}}^{\infty}\frac{\tau(|M|^{*})^{2}}{2}d\tau\leq 1.\]
	Then, from \cite[Lemma 4.1]{Bartnik:1993er}, we have
	\[1-\frac{C_4}{t}\leq\frac{1}{t}\int_{1}^{t}\exp\left(-\int_{s}^{t}\frac{\tau(|M|^*)^2}{2}d\tau\right)\leq 1+\frac{C_4}{t}.\]
	Then we get
	\[\begin{split}\delta_{*}(t) & \geq\frac{1}{1+t^{2}}-\frac{C_4}{t(1+t^{2})}+\frac{C_2}{t(1+t^{2})}\\
	&\qquad\qquad+\frac{1}{t(1+t^{2})}\int_{1}^{t}3s^{2}\exp\left(-\int_{s}^{t}\frac{\tau(|M|^{*})^{2}}{2}d\tau\right)ds+O(t^{-3})\\
	& \geq\frac{1}{1+t^{2}}-\frac{C_4}{t(1+t^{2})}+\frac{C_2}{t(1+t^{2})}+\frac{1}{t(1+t^{2})}\left[\int_{1}^{t_{0}}3s^{2}\exp\left(-\int_{s}^{t}\frac{\tau(|M|^{*})^{2}}{2}d\tau\right)ds\right.\\
	& \qquad\qquad\qquad\qquad\left.+t^{3}-t_{0}^{3}+\int_{t_{0}}^{t}3s^{2}\left(\exp\left(-\int_{s}^{t}\frac{\tau(|M|^{*})^{2}}{2}d\tau\right)-1\right)ds\right]+O(t^{-3})\\
	&\geq 1-\frac{C_4+t_0^3-C_2}{t(1+t^{2})}-\frac{1}{t(1+t^{2})}\left[\int_{t_{0}}^{t}3s^{2}\int_{s}^{t}\tau(|M|^{*})^{2}d\tau ds\right]+O(t^{-3}).\\
	\end{split}\]
    We used the fact that $e^{\eta}-1\geq-2|\eta|\text{ for }|\eta|\leq 1$ to get the third inequality. It follows from the fact that \[|M|^2\leq C_5e^{-ct}\] that we have
    \[\begin{split}
    \int_{t_0}^{t}3s^2\int_s^t \tau(|M|^*)^2 d\tau ds &\leq\int_{t_0}^{t}3s^2\int_s^t \tau C_5 e^{-c\tau} d\tau ds\\
    &\leq C_5\int_{t_0}^{t}3s^2\int_s^t \tau e^{-c\tau} d\tau ds\\
    &\leq O(e^{-ct})+C_5\int_{t_0}^{t}3s^2 \left(\frac{1}{c}se^{-cs}+\frac{1}{c^2}e^{-cs}\right) ds=O(e^{-ct}).
    \end{split}
    \]
    Hence we obtain
    \[
    \begin{split}
    \delta_{*}(t)
    & \geq1-\frac{C_6}{t^3},
    \end{split}
    \]
    and thus
	\[|u(t,x)-1|\leq\frac{C}{t^3}.\]
	Now we find the estimate for derivatives. Write the equation (\ref{targetPDE}) as follows
	\[\begin{split}
	&t(1+t^{2})\frac{\partial u}{\partial t}\\
	&\qquad=\frac{\partial}{\partial x^{i}}\left(\frac{1}{2}u^{2}g^{ij}\frac{\partial u}{\partial x^{j}}\right)-\left[\frac{\partial}{\partial x^{i}}\left(\frac{u^{2}}{2\sqrt{|g|}}\right)\right]\left(\sqrt{|g|}g^{ij}\frac{\partial u}{\partial x^{j}}\right)\\
	& \qquad\qquad-\frac{u^{3}}{4}(R_{g(t)}-t^{2}\overline{R})+u\left(\frac{1+3t^{2}}{2}+\frac{t^{2}(1+t^{2})|M|^{2}}{4}\right)\\
	&\qquad =\frac{\partial}{\partial x^{i}}\left(\frac{1}{2}u^{2}g^{ij}\frac{\partial u}{\partial x^{j}}\right)-\left(ug^{ij}\frac{\partial u}{\partial x^{i}}\frac{\partial u}{\partial x^{j}}\right)\\
	&\qquad\qquad-\left[\frac{\partial}{\partial x^{i}}\left(\frac{1}{2\sqrt{|g|}}\right)\right]\left(u^{2}\sqrt{|g|}g^{ij}\frac{\partial u}{\partial x^{j}}\right) -\frac{u^{3}}{4}(R_{g(t)}-t^{2}\overline{R})\\
	&\qquad\qquad+u\left(\frac{1+3t^{2}}{2}+\frac{t^{2}(1+t^{2})|M|^{2}}{4}\right).\\
	\end{split}\]
	Then by letting $s=\log \left(\frac{t}{\sqrt{1+t^2}}\right)+1$, we get the following form 
	\begin{equation}
	\frac{\partial u}{\partial s}=\frac{\partial}{\partial x^{i}}a_{i}(x,s,u,Du)-a(x,s,u,Du)
	\end{equation} 
	where \[a_{i}(x,s,u,p)=\frac{1}{2}u^{2}g^{ij}p_{j},\]
	\[\begin{split}
	a(x,s,u,p)=ug^{ij}p_{i}&p_{j}+\left[\frac{\partial}{\partial x^{i}}\left(\frac{1}{2\sqrt{|g|}}\right)\right]\left(u^{2}\sqrt{|g|}g^{ij}p_{j}\right)\\ &+\frac{u^{3}}{4}(R_{g(t)}-t^{2}\overline{R})-u\left(\frac{1+3t^{2}}{2}+\frac{t^{2}(1+t^{2})|M|^{2}}{4}\right).
	\end{split}\]
	It follows from the $C^0$ estimate of $u$ that\[a_{i}p_{i}\geq C|p|^{2},\quad|a_{i}|\leq C|p|,\quad|a|\leq C(1+|p|^{2}),\] where $C$ is independent of $s$. By \cite[Theorem V.1.1]{Ladyzhenskaya:1988wu}, for any $s_0,s_1\in[1-\frac{1}{2}\log 2,1)$ with $s_0<s_1$, there are constants $\beta>0$ and $C_1>0$ independent of $s_0,s_1$, such that
	\[||u||_{\beta,\beta/2;A_{[s_0,s_1]}}\leq C_1.\]
	Now consider the function $v=u-1$, we have the linear parabolic equation in terms of $v$
	\[\begin{split}
	\frac{\partial v}{\partial s}&=\frac{u^{2}}{2}g^{ij}\frac{\partial^{2}v}{\partial x^{i}\partial x^{j}}-\frac{u^{2}}{2}\frac{\partial}{\partial x^{i}}\left(\sqrt{|g|}g^{ij}\right)\frac{\partial v}{\partial x^{j}}\\
	&\qquad-\frac{u^{3}}{4}(R_{g(t)}-t^{2}\overline{R})+u\left(\frac{1+3t^{2}}{2}+\frac{t^{2}(1+t^{2})|M|^{2}}{4}\right)\\
	&:=Lv-\frac{1}{4}(R_{g(t)}-t^2\overline{R})+\frac{1+3t^2}{2}+\frac{t^2(1+t^2)|M|^2}{4}\\
	&=Lv+f,
	\end{split}\]
	where $f(x,t)=-\frac{R_{g(t)}-2}{4}+\frac{t^2(1+t^2)|M|^2}{4}+O(t^{-3})=O(t^{-3})$, since $|R_{g(t)}-2|$ and $|M|$ converge to $0$ exponentially fast. Therefore the usual Schauder interior estimates \cite[Theorem IV.10.1]{Ladyzhenskaya:1988wu} and bootstrap argument give the desired result.
\end{proof}

\begin{proof}[Proof of Theorem \ref{main theorem}]
	It suffices to show that the metric $\overline{g}$ obtained from Theorem \ref{global existence theorem} is asymptotically hyperbolic. From Lemma \ref{decay Lemma}, we have the following expression of the metric $\overline{g}$:
	\begin{equation*}
	\begin{split}
	\overline{g}&=\frac{u^2}{1+t^2}dt^2+t^2g(t)\\
	&=\frac{dt^2}{1+t^2}+O(t^{-5})dt^2+t^2g(t).
	\end{split}
	\end{equation*}
	This implies that
	\begin{equation*}
	\begin{split}
	\overline{g}_{tt}&=\frac{1}{t^2}-\frac{1}{t^4}+\frac{\overline{g}_{tt}^{(-5)}}{t^5}+\frac{\overline{g}_{tt}^{(-6)}}{t^6}+O(t^{-7}),\\
	\overline{g}_{ij}&=t^2\sigma_{ij}+O(e^{-ct})
	\end{split}
	\end{equation*}
	where $\sigma_{ij}$ is the standard metric on the sphere $S^2$ and $\overline{g}_{tt}^{(-5)},\overline{g}_{tt}^{(-6)}\in C^\infty(\Sigma)$. By adopting the definition in \cite{Chen:2016hm}, $\overline{g}$ is asymptotically hyperbolic.
\end{proof}

\begin{corollary}
	Let $(\Sigma,\sigma)$ be the $2$-sphere with the standard metric, and fix any $0<m<1$. Then by prescribing the scalar curvature $\overline{R}\equiv -6$ on $N=[1,\infty)\times\Sigma$, the metric $\overline{g}$ obtained from Theorem \ref{main theorem}, with the initial condition for the constant mean curvature $H$ on $\{1\}\times\Sigma$ as
	\[H\equiv \sqrt{8(1-m)},\]
	is the anti-de Sitter Riemannian Schwarzschild metric with the mass $m$.
\end{corollary}
\begin{proof}
	Note that from the initial metric $(\Sigma,\sigma)$ the solution to Hamilton's modified Ricci flow is constant, i.e., $|M|\equiv 0$ and $R_{g(t)}\equiv 2$. Then from (\ref{constant K}), we have
	\[K=\sup_{1\leq t<\infty}\left\{-\int_{1}^{t}\frac{2+6s^2}{4}ds\right\}=0.\]
	Thus by Theorem \ref{main theorem}, there exists an asymptotically hyperbolic metric $\overline{g}$ with mean curvature $H=\sqrt{8(1-m)}$ on $\{1\}\times\Sigma$. It is easy to see that from Proposition \ref{proposition for delta} we have 
	\[u^{-2}(t,x)=1-\frac{1}{t(1+t^2)}\left(2-\frac{H^2}{4}\right),\]
	and hence the metric on $N$ we obtained is
	\[\begin{split}
	\overline{g}&=\left(1+t^2-\frac{1}{t}\left(2-\frac{H^2}{4}\right)\right)^{-1}dt^2+t^2\sigma\\
	&=\left(1+t^2-\frac{2m}{t}\right)^{-1}dt^2+t^2\sigma
	\end{split}
	\]
	Notice that the boundary at $t=1$ is not totally geodesic. However, once we obtain the explicit form, we can extend this metric on $N=[1,\infty)\times\Sigma$ up to the totally geodesic boundary as $\overline{N}=[t_0,\infty)\times\Sigma$ where $t_0$ is the largest zero of the polynomial $t^3+t-2m$.
\end{proof}

\section{Rigidity and Monotonicity of the Hawking Mass}\label{hawking mass section}
In this section we prove Theorem \ref{rigidity theorem} regarding rigidity and monotonicity of the Hawking mass with the foliation we used in previous sections. 
The proof basically follows an argument in \cite[Theorem 5]{Lin:2016ca}.

\begin{proof}[Proof of Theorem \ref{rigidity theorem}]
	Consider $N=[1,\infty)\times\Sigma$ equipped with the metric
	\begin{equation*}
	\overline{g}=\frac{u^2}{1+t^2}dt^2+t^2g(t)
	\end{equation*}
	where $g(t)$ is the solution of the modified Ricci flow.
	From (\ref{hawking mass relation}) we have
	\begin{equation*}
	\mathbf{M}(\overline{g})(V^{(0)})=\lim_{t\rightarrow\infty}\tilde{\mathfrak{m}}_H(\Sigma_{t}).
	\end{equation*}
	We compute the Hawking mass of $\Sigma_{t}=\{t\}\times\Sigma$
	\begin{equation}\label{hawking mass formula}
	\begin{split}
	\tilde{\mathfrak{m}}_H(\Sigma_{t})&=\sqrt{\frac{|\Sigma_{t}|}{16\pi}}\left(1-\frac{1}{16\pi}\int_{\Sigma_{t}}H^2\, d\sigma_t+\frac{|\Sigma_{t}|}{4\pi}\right)\\
	&=\sqrt{\frac{4\pi t^2}{16\pi}}\left(1-\frac{1}{16\pi}\int_{\Sigma}\frac{4(1+t^2)}{t^2u^2}t^2\, d\sigma+\frac{4\pi t^2}{4\pi}\right)\\
	&=\frac{t(1+t^2)}{2}\left(1-\frac{1}{4\pi}\int_{\Sigma}u^{-2}d\sigma\right)\\
	&=\frac{1}{4\pi}\int_{\Sigma}\frac{t(1+t^2)}{2}(1-u^{-2})d\sigma.
	\end{split}
	\end{equation}
	Hence we have the following equality
	\begin{equation}
	\mathbf{M}(\overline{g})(V^{(0)})=\lim_{t\rightarrow\infty}\frac{1}{4\pi}\int_{\Sigma}\frac{t(1+t^2)}{2}(1-u^{-2})d\sigma.
	\end{equation}
	Now by Gauss-Bonnet theorem, we have
	\begin{equation*}
	\begin{split}
	\frac{d}{dt}\tilde{\mathfrak{m}}_H(\Sigma_{t})&=\frac{1}{4\pi}\int_{\Sigma}\frac{3t^2+1}{2}(1-u^{-2})+\frac{t(1+t^2)}{2}2u^{-3}\frac{\partial u}{\partial t}d\sigma\\
	&=\frac{1}{4\pi}\int_{\Sigma}\frac{3t^2+1}{2}+\frac{u^{-1}\Delta u}{2}-\frac{R}{4}+\frac{t^2\overline{R}}{4}+\frac{t^2(1+t^2)}{4u^2}|M|^2 d\sigma\\
	&=\frac{1}{4\pi}\int_{\Sigma}\frac{(\overline{R}+6)t^2}{4}+\frac{u^{-1}\Delta u}{2}+\frac{t^2(1+t^2)}{4u^2}|M|^2 d\sigma\\
	&=\frac{1}{8\pi}\int_{\Sigma}\frac{(\overline{R}+6)t^2}{2}+\frac{|\nabla u|^2}{u^2}+\frac{t^2(1+t^2)}{2u^2}|M|^2 d\sigma\geq 0\\
	\end{split}
	\end{equation*}
	given $\overline{R}\geq -6$. Thus the condition $\mathbf{M}(\overline{g})(V^{(0)})=\tilde{\mathfrak{m}}_H(\Sigma)$ implies that $\frac{d}{dt}\tilde{\mathfrak{m}}_H(\Sigma_{t})=0$, that is, $\overline{R}=-6,|M|=0,$ and $\nabla u=0$. It follows from $|M|=0$ that $(\Sigma,g_1)$ is isometric to a standard sphere. Since $\nabla u=0$, $N$ is rotationally symmetric. From the result \cite[Theorem 3.3]{Sakovich:2017jo} by Sakovich and Sormani, if $\tilde{\mathfrak{m}}_H(\Sigma)=0$ then $N$ is isometric to a hyperbolic space or if $\tilde{\mathfrak{m}}_H(\Sigma)=m>0$ then $N$ is isometric to a Riemannian anti-de Sitter Schwarzschild manifold of mass $m$.
\end{proof}
\bibliography{untitled}
\end{document}